\documentclass[runningheads]{llncs}
\usepackage[T1]{fontenc}
\usepackage{graphicx}
\usepackage{amsmath}
\usepackage{amsfonts}
\usepackage{amssymb}
\usepackage{epsfig}
\usepackage{graphicx}
\usepackage{etex}
\usepackage{url}
\usepackage{bbold}
\usepackage[usenames,dvipsnames]{xcolor}
\usepackage[colorlinks=true,linkcolor=Blue,citecolor=Green,urlcolor=black]{hyperref}
\usepackage{nicefrac}
\usepackage{aliascnt}		
\usepackage[alwaysadjust]{paralist}
\usepackage[T1]{fontenc}
\usepackage[utf8]{inputenc}
\usepackage{tikz}
\usetikzlibrary{calc}
\usepackage{enumerate}
\usetikzlibrary{arrows}
\usetikzlibrary{positioning}
\definecolor{myred}{RGB}{163, 51, 61}
\definecolor{myblue}{RGB}{90, 107, 127}
\definecolor{mygreen}{RGB}{133, 177, 168}

\usepackage[numbers]{natbib}
\usetikzlibrary{shapes,calc,arrows,trees,positioning,through,intersections,fadings}
\usetikzlibrary{decorations.markings}
\usetikzlibrary{patterns}

\newcommand{\Z}{\mathbb{Z}}

\newcommand{\R}{\mathbb{R}}

\DeclareMathOperator{\vol}{vol}

\DeclareMathOperator{\poly}{poly}

\newcommand{\norm}[1]{\|#1\|}

\newcommand{\df}{\mathrel{\mathop:}=}

\newcommand{\cF}{\mathcal{F}}

\DeclareMathOperator{\supp}{supp}

\usepackage{thmtools}
\usepackage{thm-restate}
\begin{document}
\title{Sparse Approximation Over the Cube}
%
%
\author{Sabrina Bruckmeier \inst{1} \orcidID{0000-0001-5673-3383} \and
Christoph Hunkenschröder \inst{2}\orcidID{0000-0001-5580-3677} \and
Robert Weismantel \inst{1}}
\authorrunning{S. Bruckmeier et al.}
%
\institute{ETH Zürich, Zürich, Switzerland \\ \email{\{sabrina.bruckmeier, robert.weismantel\}@ifor.math.ethz.ch} \and
TU Berlin, Berlin, Germany\\ \email{hunkenschroeder@tu-berlin.de}}
\maketitle              
\begin{abstract}
This paper presents an anlysis of the NP-hard minimization problem $\min \{\|b - Ax\|_2: \ x \in [0,1]^n, | \supp(x) | \leq \sigma\}$, where $\supp(x) \df \{i \in [n]: \ x_i \neq 0\}$ and $\sigma$ is a positive integer.
The object of investigation is a natural relaxation where we replace $| \supp(x) | \leq \sigma$ by $\sum_i x_i \leq \sigma$.
Our analysis includes a probabilistic view on when the relaxation is exact.
We also consider the problem from a deterministic point of view and provide a bound on the distance between the images of optimal solutions of the original problem and its relaxation under $A$.
This leads to an algorithm for generic matrices $A \in \Z^{m \times n}$ and achieves a polynomial running time
provided that $m$ and $\|A\|_{\infty}$ are fixed.

\keywords{Sparse Approximation  \and Subset Selection \and Signal Recovery.}
\end{abstract}
\section{Introduction and Literature Review}
Due to the recent development of machine learning, data science and signal processing, more and more data is generated, but only a part of it might be necessary in order to already make predictions in a sufficiently good manner.
Therefore, the question arises to best approximate a signal $b$ by linear combinations of no more than $\sigma$ vectors $A_i$ from a suitable dictionary 
$A=\begin{pmatrix}
A_1, \dots, A_n
\end{pmatrix} \in \mathbb{R}^{m \times n}$:
\begin{equation}
\label{Problem Introduction}
\min \norm{Ax-b}_2 \text{ subject to } \norm{x}_0 \leq \sigma,
\end{equation}
where $\norm{x}_0:=|\lbrace i \in \lbrack n \rbrack: x_i \neq 0 \rbrace|$.
Additionally, many areas of application -- as for example portfolio selection theory, sparse linear discriminant analysis, general linear complementarity problems or pattern recognition -- require the solution $x$ to satisfy certain polyhedral constraints.
While there exists a large variety of ideas how to tackle this problem, the majority of them relies on the matrix $A$ satisfying conditions such as being sampled in a specific way or being close to behaving like an orthogonal system, that might be hard to verify.
Additionally, these algorithms commonly yield results only with a certain probability or within an approximation factor that again highly depends on the properties of $A$.
A discussion of these ideas and different names and variants of this problem is postponed to the end of the introduction.
In this work, we develop an exact algorithm that, without these limitations on $A$, solves the \emph{Sparse Approximation problem} in $[0,1]$-variables:
\begin{equation}
\label{P0}
\tag{$P_0$}
\min_x \| Ax-b \|_2
\text{ subject to } x \in [ 0,1 ]^n \text{ and }
\norm{x}_0 \leq \sigma.
\end{equation}
\begin{restatable}{theorem}{theomain}\label{thm:main-1}
Given $A \in \mathbb{Z}^{m \times n}, b \in \mathbb{Z}^m$ and  $\sigma \in \mathbb{Z}_{\geq 1}$, we can find an optimal solution $x$ to Problem~\eqref{P0} in $(m \norm{A}_\infty)^{\mathcal{O}(m^2)} \cdot \poly(n, \|b\|_1)$ arithmetic operations.
\end{restatable}
Relaxing the pseudonorm $\norm{\cdot}_0$ by $\norm{\cdot}_1$ is a commonly used technique in the literature.
In contrast to previous results we are able to bound the distance between the images of these solutions under $A$ without any further assumptions on the input data and therefore derive a \emph{proximity result} that -- to the best of our knowledge -- has not been known before.
\begin{restatable}{theorem}{theoprox}
\label{thm:proximity}
Let $\hat{x}$ be an optimal solution to the following relaxation of~\eqref{P0}:
\begin{equation*}
\min_x \norm{Ax-b}_2 
\text{ subject to } x \in \lbrack 0,1 \rbrack^n \text{ and }
\norm{x}_1 \leq \sigma.
\end{equation*}
For an optimal solution $x^\star$ to~\eqref{P0} we have
\begin{equation*}
\norm{Ax^\star - A\hat{x}}_2 \leq 2\norm{\hat{x} - \lfloor \hat{x} \rfloor}_1 \max_{i=1,\dots,n} \norm{A_i}_2 \leq 2 m^{3/2} \norm{A}_{\infty},
\end{equation*}
where $\lfloor \hat{x} \rfloor$ denotes the vector $\hat{x}$ rounded down component-wise. 
\end{restatable}
We also illuminate our approach from a probabilistic point of view.
Specifically, the hard instances are those where $b$ is relatively close to the boundary of the polytope $Q \df \{Ax: \ x \in [0,1]^n, \ \|x\|_1 \leq \sigma\}$.
Conversely, if $b$ is deep inside $Q$ or far outside of $Q$, then with high probability, an optimal solution to the relaxation solves the initial problem~\eqref{P0}.

The paper is organized as follows.
We conclude the introduction by providing an overview on related literature.
Section~\ref{sec:prelim} discusses preliminaries.
The probabilistic analysis of a target vector $b$ is carried out in Section~\ref{sec:sampling}.
We then discuss a worst-case proximity bound between optimal solutions of~\eqref{P0} and a natural relaxation in Section~\ref{sec:proximity}.
This will allow us to formalize a deterministic algorithm in Section~\ref{sec:algorithm}.

In the literature, Problem \eqref{Problem Introduction} can be found under various modifications and names, see e.g.\ \cite{DecodingByLinearProgramming, CandesTao, Bresler, Samet}.
A common variant in the context of random measurements is often called \textit{Sparse Recovery}, cf.~\cite{Gilbert}, or \textit{Subset Selection for (linear) regression}, cf.~\cite{Das}, while the name \textit{(Best) Subset Selection} is generally used without further interpretation cf.~\cite{Bresler, Weismantel, Zhu}, in contrast to \textit{Signal Recovery} or \textit{Signal Reconstruction} as in \cite{Romberg}. 
If the vector $b$ can be represented exactly, the problem is called \textit{Exact Sparse Approximation} or \textit{Atomic Decomposition}, cf.~\cite{ChenDonohoSaunders, Gilbert02approximationof, Tropp, Civril}.
Since the differences are marginal and the names in the literature not well-defined, we restrain ourselves to the name \textit{Sparse Approximation} for simplicity. In general, there are two common strategies used to tackle \textit{Sparse Approximation}: Greedy algorithms and relaxations -- a detailed discussion of which would be beyond the scope of this paper. These algorithms either recover the optimal support only under certain conditions (compare \cite{ Ament, Bresler, Kempe, Tropp}), recover it with high probability (see for example \cite{Donoho, Zhu}) or approximate the solution (for instance \cite{Das, Ethan, Gilbert02approximationof}). 
Unfortunately, because of their high computational cost most common greedy algorithms are not sufficient for large systems, though experiments suggest that there still exist applicable greedy approaches, such as the Dropping Forward-Backward Scheme, introduced by Nguyen \cite{Nguyen}.
While the idea of relaxing the pseudonorm $\norm{\cdot}_0$ by the norm $\norm{\cdot}_1$, as done for example in Basis Pursuit by Chen, Donoho and Saunders \cite{ChenDonohoSaunders}, might seem intuitive, for a long time the success of this method was not quite understood. This changed as Candes, Romberg and Tao \cite{DecodingByLinearProgramming, Candes} discovered and improved the Uniform Uncertainty Principle. For the usually problematic case of having not enough data points, the Dantzig Selector presented by Candes and Tao \cite{CandesTao} yields a sophisticated estimator with high probability. Similarly, LASSO based methods, see for instance \cite{Samet}, either recover the support with high probability exactly under certain conditions, or fail with high probability if the conditions are not met, cf.~\cite{Wainwright}.  Finally, Garmanik and Zadik \cite{David} revealed interesting structural results, that explain the above mentioned all-or-nothing behavior.
There also exists a series of papers in a similar line of thought that relaxes $\norm{\cdot}_0$ by smooth, non-decreasing, concave functions, see \cite{RinaldiSchoen, Waechter,Fung, Dongdong,Haddou,Mangasarian,  Rinaldi,  Schoen}. It can be shown that these relaxations converge towards the optimal solution of \eqref{P0}. Qian et al. \cite{Chao} and Çivril \cite{Civril} proved that, unless $P=NP$, for a general matrix $A$ Pareto Optimization and the two greedy algorithms, Forward Selection and Orthogonal Matching Pursuit, are almost the best we can hope for. This motivated a search for more efficiently solvable classes of $A$, cf. \cite{Romberg, Weismantel, Gao, Gilbert}. 
Finally, it should be mentioned that there exists a variety of Branch-and-Bound algorithms whose success though is in general only tested experimentally, see \cite{Kendall, Hocking}. 

\section{Preliminaries}
\label{sec:prelim}
Let $A \in \mathbb{Z}^{m \times n}$ and $b \in \mathbb{Z}^m$.
Moreover, let $\supp (x)$ denote the support of $x$, i.e.\ $\supp (x) \df \{ i \in [ n ] : x_i \neq 0 \}$ and set $\norm{x}_0  \df |\supp (x)|$.
For the rest of the paper, $x^\star$ denotes an optimal solution for \eqref{P0} for a given integer $\sigma \in \Z_{\geq 1}$.
A natural convex relaxation of~\eqref{P0} is given by
\begin{equation}
\label{P1}\tag{$P_1$}
\min_x \norm{Ax-b}_2 \text{ subject to } x \in \lbrack 0,1 \rbrack^n \text{ and }
\norm{x}_1 \leq \sigma.
\end{equation}
An optimal solution to~\eqref{P1} will be denoted by $\hat{x}$ throughout the paper.
When $m=1$, there exists an optimal solution $\hat{x}$ for~\eqref{P1} that has at most one fractional variable (see Lemma~\ref{lem:frac-entries}).
This solution is also feasible for~\eqref{P0}, and hence optimal.
The idea of our approach is to establish a proximity result for 
$A\hat{x}$ and $Ax^\star$ respectively,
that we can exploit algorithmically.
This proximity bound depends on $m$ which comes as no surprise, given that the problem is NP-hard even for fixed values of $m$. 
The latter statement can be verified by reducing the NP-hard partition problem to an instance of~\eqref{P0}.
\begin{theorem}
The problem \eqref{P0} is NP-hard, even if $m=2$.
\end{theorem}
A simple but important ingredient of our proximity theorem is the following fact that can be derived from elementary linear programming theory.
\begin{lemma}[Few fractional entries]
\label{lem:frac-entries}
\leavevmode
\begin{enumerate}
\item Let $x$ be a feasible point for~\eqref{P0}.
There exists a solution $x^\prime$ such that $Ax = Ax^\prime$ with at most $m$ fractional entries. \label{pt:1}
\item Let $x$ be a feasible point for~\eqref{P1}.
There exists a solution $x^\prime$ such that $Ax = Ax^\prime$ with at most $m$ fractional entries.
\end{enumerate}
\end{lemma}
\begin{proof}
\begin{enumerate}
\item Let $x$ be a solution of~\eqref{P0} and denote $S = \supp(x)$.
Let $A_S$ denote the submatrix of $A$ comprising the columns with indices in $S$.
The set
\begin{equation*}
P_S(x) \df \lbrace y \in \mathbb{R}^{|S|}: \ A x = A_S y, \ 0 \leq y \leq 1 \rbrace
\end{equation*}
is a polytope.
It is non-empty since $x \in P$, hence it has at least one vertex $v$.
By standard LP theory, at least $|S|-m$ inequalities of the form $0 \leq y \leq 1$ are tight at $v$.
It follows that $v$ has at most $m$ fractional entries.
The vertex $v$ can easily be extended to a solution $x^\prime$ of~\eqref{P0} by adding zero-entries. 
\item Given the solution $x$ to~\eqref{P1}, consider the optimization problem
\begin{equation*}
\min \{ \sum_{i=1}^n y_i : \ y \in P_{ \{ 1,\dots,n \}}(x) \}.
\end{equation*}
Let $v$ be an optimal vertex solution.
From Part~\ref{pt:1} $v$ has at most $m$ fractional entries.
Since $x$ is feasible for the above problem, $\sum_{i=1}^n v_i \leq \sum_{i=1}^n x_i \leq \sigma$.
\end{enumerate}
\end{proof}

\section{The \texorpdfstring{$\ell_1$}{l_1}-Relaxation for Random Targets \texorpdfstring{$b$}{b}}
\label{sec:sampling}
In order to shed some light on Problem~\eqref{P0} and its natural convex relaxation~\eqref{P1} we first provide a probabilistic analysis to what extend optimal solutions of~\eqref{P1} already solve~\eqref{P0}.
Let $Q \df \{Ax \in \R^m : \  x \in [0,1]^n, \|x\|_1 \leq \sigma\}$ be the set of all points we can represent with the $\ell_1$-relaxation.
This section deals with the question which vectors $b$ are ``easy'' target vectors.
It turns out that if $b$ is ``deep'' inside $Q$ or far outside of $Q$, then the corresponding instances of~\eqref{P0} are easy with very high probability.
In fact, there almost always exist optimal solutions of~\eqref{P1} that are already feasible for~\eqref{P0} and hence optimal.
Conversely, if $b$ is close to the boundary of $Q$, then the probability that an optimal solution of~\eqref{P1} solves~\eqref{P0} is almost $0$.

\begin{theorem}
\label{thm:inside-Q}
Let $A \in \R^{m \times n}$
and $\sigma \geq m$ be an integer.
If $b \in \tfrac{\sigma - m + 1}{\sigma} Q$, then there exists $x^\star \in [0,1]^n$ with $\norm{x^\star}_0 \leq \sigma$ and $Ax^\star = b$.
\end{theorem}
\begin{proof}
If $b \in \tfrac{\sigma - m + 1}{\sigma} Q$,
there exists a vector $\hat{x} \in [0,\tfrac{\sigma-m + 1}{\sigma}]^n$ such that $b = A\hat{x}$ and $\norm{\hat{x}}_1 \leq \sigma-m + 1$.
Let $v$ be a vertex of $\lbrace x \in [0,1]^n:  Ax = b,\norm{x}_1 \leq \sigma - m + 1\rbrace$, which contains $\hat{x}$.
According to the constraint $\norm{x}_1 \leq \sigma - m + 1$, $v$ has at most $\sigma - m + 1$ integral non-zero entries.
By Lemma~\ref{lem:frac-entries},
$v$ has at most $m$ fractional entries.
However, if there are fractional entries present, we can only have $\sigma-m$ integral entries,
thus, $\|v\|_0 \leq \sigma$.
\end{proof}

For the next result, we require some theory of \emph{mixed volumes}.
We will restrict to stating results necessary for our work, and refer to~\cite[Sec.\ $6$]{gruber2007convex} and references therein for more background.
Let $B \df B(0,1) \subseteq \R^m$ be the Euclidean ball of radius $1$, $Q \subseteq \R^m$ be a convex body and $\lambda > 0$.
The following equality is known as \emph{Steiner's formula}, and shows that the volume of $Q + \lambda B$ is a polynomial in $\lambda$ (\cite[Thm.~$6.6$]{gruber2007convex}):
\begin{align}
\label{eq:steiner}
\vol (Q + \lambda B) &= \sum_{i=0}^m \binom{m}{i} W_i(Q) \lambda^i, 
\end{align}
where the constants $W_i(Q)$ are called the \emph{Quermassintegrals} of $Q$.
Let $\mu > 0$ be a constant s.t.\ $Q \subseteq \mu B$.
Then,
$W_i(Q) \leq \mu^{m - i} \vol(B)$.
Combining the two inequalities, we have
\begin{align}
\label{eq:volume}
\vol (Q + \lambda B) &\leq \vol(B) \sum_{i=0}^m \binom{m}{i} \mu^{m-i} \lambda^i = \vol(B) (\lambda + \mu)^m
\end{align}
for any convex body $Q \subseteq \mu B$.
We are now prepared to show the following result.
\begin{theorem}
\label{thm:b-far}
Let $A \in \mathbb{R}^{m \times n}$,
and $\sigma \geq 1$ be an integer.
If $b$ is sampled uniformly at random from the convex set
$Q + \lambda B$, then with probability at least 
\begin{equation*}
\left( \frac{\lambda}{\lambda + \sigma \sqrt{m} \norm{A}_{\infty}} \right)^m 
\end{equation*}
there exists an optimal solution of~\eqref{P1}
that is optimal for~\eqref{P0}.
\end{theorem}
\begin{proof}
Define $P \df \{x \in [0,1]^n : \ \norm{x}_1 \leq \sigma \}$, and set $Q = \{Ax: \ x \in P\}$.
Observe that all vertices of $P$ are in $\{0,1\}^n$, and as a consequence any vertex $v$ of $Q$ can be written as 
\begin{equation}
\label{eq:int-vertex}
v = Ax \text{ with } x \text{ a vertex in } P \text{ that is integral.}
\end{equation}
Hence, whenever an optimal solution to $\min \{\|b-x\|_2: \ x \in Q\}$ is attained by a vertex of $Q$, the problem~\eqref{P1} has an optimal integral vertex solution $v$. 
Since an integral solution to~\eqref{P1} is also feasible for~\eqref{P0}, the vector $v$ is also optimal for~\eqref{P0}.

Let $V$ be the vertex set of $Q$.
For $v \in V$, denote the normal cone of $v$ by 
\begin{equation*}
C_v \df \{c \in \R^m: \ c^\intercal (w-v) \leq 0 \ \forall w \in Q \}.
\end{equation*}
Fix a vertex $v$ and assume $b \in v + C_v$.
We next show that $v$ is an optimal solution to $\min\{\|b-x\|^2_2: \ x \in Q\}$.
Since $b = v + c$ with $c^\intercal (v-w) \geq 0$ for all $w \in Q$, we obtain
\[
\norm{b-w}^2_2 = \norm{v-w+c}_2^2 = \norm{v-w}^2_2 + \norm{c}^2_2 + 2 c^\intercal (v-w) \geq \norm{c}^2_2 = \norm{b-v}^2_2,
\]
showing that $v$ is optimal.
By Equation~\eqref{eq:int-vertex} there exists an integral $x \in P$ such that $v = Ax$ and hence $x$ is optimal for~\eqref{P0}.
It remains to calculate the probability that $b \in v + C_v$ for some vertex $v$ of $Q$.
We obtain
\begin{align*}
\vol \left( \bigcup_{v \in V} (v + C_v) \cap (Q + \lambda B) \right)
&= \vol \left( \bigcup_{v \in V} C_v \cap \lambda B \right) \\
&= \vol \left( \left( \bigcup_{v \in V} C_v \right) \cap \lambda B \right) \\
&= \vol(\lambda B) = \lambda^m \vol(B).
\end{align*}
In the second to last equality we used that the normal cones $C_v$ tile the space $\R^m$.

Let $\mu > 0$ be a constant s.t.\ $Q \subseteq \mu B$, e.g.\ $\mu = \sigma \sqrt{m} \norm{A}_{\infty}$.
We now can apply Steiner's Formula~\eqref{eq:steiner} to estimate $\vol(Q + \lambda B) \leq (\lambda + \mu)^m \vol(B)$.
The probability that $b$ is sampled in one of the normal cones is therefore
\begin{align*}
\frac{\vol(\lambda B)}{\vol (Q + \lambda B)} \geq \frac{\lambda^m}{(\lambda + \mu)^m}
\geq \left( \frac{\lambda}{\lambda + \sigma \sqrt{m} \norm{A}_{\infty}} \right)^m. 
\end{align*}
\end{proof}
Let us briefly comment on the probability quantity $\rho \df ( \lambda / (\lambda + \sigma \sqrt{m} \norm{A}_{\infty}) )^m$.
If we choose $\lambda = 2m^{3/2}\sigma \|A\|_{\infty}$ in Theorem~\ref{thm:b-far}, then $\rho \geq 1/2$, as one can verify with Bernoulli's inequality. 
Figure \ref{figure quality of target} depicts the geometry underlying the proof of Theorem \ref{thm:b-far}.
The vector $b_1$ is sampled from the dotted area and hence, an optimal solution of~\eqref{P1} may use $2$ fractional entries, and thus have support $\sigma+1$.
On the other hand, the vector $b_2$ is sampled from the dashed area, which leads to the solution of~\eqref{P1} corresponding to a vertex of $Q$.
In the second case~\eqref{P1} has an integral solution, which automatically solves~\eqref{P0}.

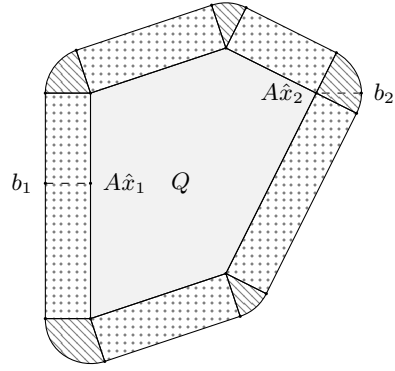
\begin{figure}
\centering
\begin{tikzpicture}[scale=0.6, vert/.style args={of #1 at #2}{insert path={%
#2 -- (intersection cs:first
  line={#1}, second line={#2--($#2+(0,10)$)}) }},
vert outwards/.style args={from #1 by #2 on line to #3}{insert path={
#1 -- ($#1!#2!90:#3$)
}}]

\coordinate (v1) at (-5,-3);
\coordinate (v2) at (-2,-2);
\coordinate (v3) at (0,2);
\coordinate (v4) at (-2,3);
\coordinate (v5) at (-5,2);
\coordinate (b1hat) at (-5,0);
\coordinate (b1) at (-6, 0);
\coordinate (M) at (1,2);

\draw[fill = gray, opacity = 0.1] (v1) -- (v2) -- (v3) -- (v4) -- (v5) -- cycle;
\draw [dashed] (b1)--(b1hat);

\node at (v3) [left = .5mm of v3] {$A\hat{x}_2$};

\draw (v1)--(v2);
\draw (v2)--(v3);
\draw (v3)--(v4);
\draw (v4)--(v5);
\draw (v5)--(v1);


\coordinate (v51) at ($(v5)!1cm!90:(v4)$);
\coordinate (v52) at ($(v5)!1cm!-90:(v1)$);

\coordinate (v41) at ($(v4)!1cm!90:(v3)$);
\coordinate (v42) at ($(v4)!1cm!-90:(v5)$);


\coordinate (v31) at ($(v3)!1cm!90:(v2)$);
\coordinate (v32) at ($(v3)!1cm!-90:(v4)$);
\draw[dashed] (M)--(v3);
\node [right = .5mm of M] at (M) {$b_2$};

\coordinate (v21) at ($(v2)!1cm!90:(v1)$);
\coordinate (v22) at ($(v2)!1cm!-90:(v3)$);

\draw[pattern=north west lines, pattern color=gray] (v32).. controls +($.16*(v32)-.16*(v41)$) and +($.16*(v31)-.16*(v22)$) .. (v31) -- (v3) -- cycle;

\coordinate (v11) at ($(v1)!1cm!90:(v5)$);
\coordinate (v12) at ($(v1)!1cm!-90:(v2)$);

\node at (-3,0) {$Q$};   

\draw[pattern=north west lines, pattern color=gray] (v51).. controls +($.1*(v51)-.1*(v42)$) and +($.1*(v52)-.1*(v11)$) .. (v52) -- (v5) -- cycle;  

\draw[pattern=north west lines, pattern color=gray] (v11).. controls +($.16*(v11)-.16*(v52)$) and +($.18*(v12)-.18*(v21)$) .. (v12) -- (v1) -- cycle;

\draw[pattern=north west lines, pattern color=gray] (v21).. controls +($.1*(v21)-.1*(v12)$) and +($.05*(v22)-.05*(v31)$) .. (v22) -- (v2) --cycle;

\draw[pattern=north west lines, pattern color=gray] (v41).. controls +($.1*(v41)-.1*(v32)$) and +($.1*(v42)-.1*(v51)$) .. (v42) -- (v4) -- cycle;

\draw[pattern = dots, pattern color = gray] (v41) -- (v4) -- (v3) -- (v32) -- cycle;

\draw[pattern = dots, pattern color = gray] (v31) -- (v3) -- (v2) -- (v22) -- cycle;

\draw[pattern = dots, pattern color = gray] (v21) -- (v2) -- (v1) -- (v12) -- cycle;

\draw[pattern = dots, pattern color = gray] (v11) -- (v1) -- (v5) -- (v52) -- cycle;

\draw[pattern = dots, pattern color = gray] (v51) -- (v5) -- (v4) -- (v42) -- cycle;

\foreach \p in {v1,v2,v3,v4, v5, b1, b1hat, v11, v12, v21, v22, v31, v32, v41, v42, v51, v52, M}
\fill (\p) circle(1pt);
\node at (b1)[ left = .5mm of b1]{$b_1$};
\node at (b1hat) [right = .5mm of b1hat] {$A\hat{x}_1$};

\end{tikzpicture}
\caption[The sampling of $b$]{The sampling of the vector $b$ from $Q+\lambda B$}
\label{figure quality of target}
\end{figure}

\section{Proximity between Optimal Solutions of~\eqref{P0} and~\eqref{P1}}
\label{sec:proximity}
In this section we illuminate the Problems~\eqref{P0} and~\eqref{P1} from a deterministic point of view and develop worst-case bounds for the distance of the images of corresponding optimal solutions under $A$.
Our point of departure is an optimal solution $\hat{x}$ of~\eqref{P1}.
The target is to show that there exists an optimal solution $x^\star$ of~\eqref{P0} satisfying $\norm{A(\hat{x} - x^\star)}_2 \leq 2m^{3/2} \norm{A}_\infty$.
Our strategy is to define a hyperplane containing $A\hat{x}$ in the space of target vectors $b$ that separates $b$ from all vectors $Ax$ with $x$ feasible for~\eqref{P0}.
The next step is to show that if we perturb $\hat{x}$ along the fractional variables, we will remain in this hyperplane.
This has the side-effect that we can find a feasible solution for~\eqref{P0} whose image is in the vicinity of $A \hat{x}$.
The triangle inequality and basic geometry then come into play to establish the claimed bound.

We introduce the hyperplane tangent to the ball $B(b,\|b-A\hat{x}\|_2)$ in $A\hat{x}$,
\begin{equation*}
H \df \{ y \in \R^m: (b - A\hat{x})^\intercal y = (b - A\hat{x})^\intercal A\hat{x} \}.
\end{equation*}
\begin{lemma}
\label{lem:sep}
We have $(b - A\hat{x})^\intercal (A x - A\hat{x}) \leq 0$ for any point $x$ feasible for~\eqref{P1}.
\end{lemma}
\begin{proof}
Assume that there exists a point $x$ feasible for~\eqref{P1} for which the inequality $(b - A\hat{x})^\intercal (Ax - A\hat{x}) > 0$ holds.
As a convex combination the point $p \df A(\hat{x} + \varepsilon (x - \hat{x}))$ is feasible for~\eqref{P1} for each $\varepsilon \in [0,1]$, and we can estimate the objective value as
\[
\| b - p \|_2^2 = \| b-A\hat{x}\|^2_2 + \varepsilon^2 \| A(x-\hat{x}) \|_2^2 - 2\varepsilon (b-A\hat{x})^\intercal (Ax - A\hat{x}) < \|b - A \hat{x} \|_2^2
\]
for $\varepsilon$ small enough.
This contradicts the optimality of $\hat{x}$.
We illustrated the argument geometrically in Figure~\ref{fig:hyperplane}.
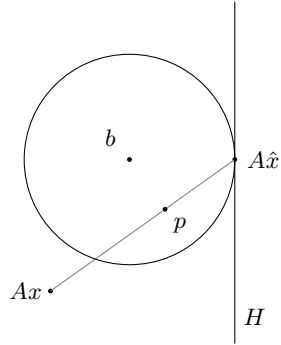
\begin{figure}
\centering
\begin{tikzpicture}[scale=0.7]
\coordinate (b) at (0, 0);
\draw[fill, black] (b) circle[radius=1pt];
\node at (b) [above left = 1mm of b] {$b$};

\coordinate (bhat) at (2,0);
\draw[fill, black] (bhat) circle[radius=1pt];
\node at (bhat) [right = .5mm of bhat] {$A\hat{x}$};

\coordinate (H) at (2, -3);
\node at (H) [right == 1mm of H] {$H$};

\coordinate (bstar) at (-1.5, -2.5);
\draw[fill, black] (bstar) circle[radius = 1pt];
\node at (bstar) [left == 1mm of bstar] {$Ax$};

\draw[gray] (bstar)--(bhat);

\coordinate (p) at ($(bstar)!(b)!(bhat)$);
\draw[fill, black] (p) circle[radius = 1pt];
\node at (p) [below right == 1mm of p] {$p$};

\draw (2,3)--(2,-3.5);

\draw (b) circle(2);

\end{tikzpicture}
\caption[The separation argument]{
If $H$ does not separate $b$ from $Ax$, there is a point $p$ closer to $b$ than $A\hat{x}$.
}
\label{fig:hyperplane}
\end{figure}
\end{proof}
An important property is that $H$ contains many points that we can easily generate from $\hat{x}$.
This is made precise below.
\begin{lemma}
\label{lem:H-affine}
Define $\mathcal{F} \df \{i \in [n]: \ \hat{x}_i \notin \Z\}$ where $\hat{x}$ is an optimal solution to~\eqref{P1}.
We have
\begin{equation*}
A\hat{x} + \left\{ \sum_{i \in \mathcal{F}} \lambda_i A_i: \ \sum_{i \in \mathcal{F}} \lambda_i = 0 \right\}
\subseteq H.
\end{equation*}
\end{lemma}
\begin{proof}
Let $v \df b - A\hat{x} \in \R^m$ be the normal vector of $H$ and let
\begin{equation*}
y = \sum_{i\in \mathcal{F}} \lambda_i e_i \quad \text{ for some } \lambda_i \in \R \text{ with } \quad  \sum_{i \in \cF} \lambda_i = 0.
\end{equation*}
Since $x_i \in (0,1)$ for all $i \in \mathcal{F}$, there exists $\varepsilon >0$ such that both points
$\hat{x} + \varepsilon y$ and $\hat{x} - \varepsilon y$ are feasible for~\eqref{P1}.
By Lemma~\ref{lem:sep}, we must have $v^\intercal A(\hat{x} +\varepsilon y - \hat{x}) = v^\intercal A \varepsilon y \leq 0$ and $-v^\intercal A \varepsilon y \leq 0$, resulting in $v^\intercal Ay = 0$.
Thus, $A(\hat{x} + y) \in H$.
\end{proof}
With these results we are now able to show a proximity result (Theorem~\ref{thm:proximity}) between $A\hat{x}$ and $Ax^\star$.
Here, $\lfloor \hat{x} \rfloor$ denotes the vector $\hat{x}$ rounded down component-wise.
\begin{proof}[Theorem~\ref{thm:proximity}]
Given an optimal solution $\hat{x}$ of~\eqref{P1}, let $\mathcal{F} = \{i \in [n]: \ x_i \in (0,1)\}$.
Without loss of generality, we may assume that $|\mathcal{F}| \leq m$ and $\mathcal{F} = \{1,2,\dots,|\mathcal{F}|\}$.
Let $k \df \sum_{i \in \mathcal{F}} \hat{x}_i$, and construct a feasible solution $y$ for~\eqref{P0} from $\hat{x}$ as follows:
\begin{equation*}
y_i:=\begin{cases}
1, & 1\leq i \leq \lfloor k \rfloor \\
k-\lfloor k \rfloor, & i = \lceil k \rceil \\
0, & \lceil k \rceil +1 \leq i \leq |\mathcal{F}|\\
\hat{x}_i, & i \notin \mathcal{F}.
\end{cases}
\end{equation*}
The point $y$ satisfies $0 \leq y_i \leq 1$ for all $i \in \lbrack n \rbrack$ and $\norm{y}_0 = \lceil \norm{\hat{x}}_1 \rceil \leq \sigma$.
Since 
\begin{equation}
\label{eq:affine-diff}
\sum_{i \in \mathcal{F}} y_i = \sum_{i \in \mathcal{F}} \hat{x}_i,
\end{equation}
Lemma~\ref{lem:H-affine} implies that $Ay \in H$ and hence
\begin{equation*}
\norm{b - Ay}_2^2 = \norm{b - A\hat{x}}_2^2 + \norm{A \hat{x} - A y}_2^2.
\end{equation*}
Assume $\norm{Ax^\star - A\hat{x}}_2 \geq \norm{Ay - A\hat{x}}_2$ holds for the optimal solution $x^\star$ for~\eqref{P0}.
Since $y$ is feasible for~\eqref{P0}, we also know $\norm{b- Ay}_2 \geq \norm{b-Ax^\star}_2$.
We are now prepared to estimate (using Lemma~\ref{lem:sep} in the third line)
\begin{align*}
\norm{b - Ax^\star}_2^2 &= \norm{b - A\hat{x} + A\hat{x}- Ax^\star}_2^2 \\
	&= \norm{b - A\hat{x}}_2^2 + \norm{A\hat{x}- Ax^\star}_2^2 + 2(b - A\hat{x})^\intercal (A\hat{x}- Ax^\star) \\
	&\geq \norm{b - A\hat{x}}_2^2 + \norm{A\hat{x}- Ax^\star}_2^2 \\
	&> \norm{b - A\hat{x}}_2^2 + \norm{A\hat{x}- Ay}_2^2 \\
	&= \norm{b-Ay}_2^2,
\end{align*}
showing that $x^\star$ is not optimal.
The proof is finished by observing that Equation~\eqref{eq:affine-diff} also implies 
$\norm{\hat{x}-y}_1 \leq 2 \norm{\hat{x} - \lfloor \hat{x} \rfloor}_1$, and consequently 
\begin{equation*}
\norm{Ay - A\hat{x}}_2 \leq 2\norm{\hat{x} - \lfloor \hat{x} \rfloor}_1 \max_{i=1,\dots,n} \norm{A_i}_2 \leq 2m^{3/2} \norm{A}_{\infty}.
\end{equation*}
\qed
\end{proof}

\section{A Deterministic Algorithm}
\label{sec:algorithm}
The results presented so far give rise to a conceptually simple algorithm.
Compute an optimal solution $\hat{x}$ to~\eqref{P1}.
According to the proximity theorem in Section~\ref{sec:proximity}, we can limit our search for an optimal right-hand side vector $b^\star = Ax^\star$ in the vicinity of $A\hat{x}$.
Since $b^\star$ might be fractional, we cannot enumerate all possible right-hand sides.
Instead, we refine our approach by decomposing $x^\star = z^\star + f^\star$ into its integral part $z^\star$ and its fractional part $f^\star$.
We first guess the support $\cF$ of the fractional entries, which satisfies $|\cF| \leq m$ by Lemma~\ref{lem:frac-entries}.
For the remaining variables, we next establish a candidate set $Z^\star$ comprising the potential vectors $z^\star$ in the decomposition of $x^\star$.
It will be essential to determine a bound on $|Z^\star|$.
This is where the proximity theorem comes into play.
We now enumerate the elements of $Z^\star$ and extend each of them by a vector $f^\star$ whose support is in the index set $\cF$ that we guessed upfront.
A composition of these two solutions will provide $x^\star$.

This section is devoted to analyze this conceptually simple algorithm and this way shed some light on some of the details required.

Before we describe the decomposition $x^\star = z^\star + f^\star$ in more detail, we discuss the standard obstacle in convex optimization that $\hat{x}$ can only be approximated.
To be more precise, we call a solution $\bar{x}$ to~\eqref{P1} \emph{$\varepsilon$-close}, if 
\begin{equation}
\label{eq:eps-close}
\|b - A \bar{x}\|_2^2 - \|b - A\hat{x}\|^2_2   \leq  \varepsilon^2. 
\end{equation}

We obtain a canonical corollary from the proximity Theorem~\ref{thm:proximity}.
\begin{corollary}
\label{cor:proximity}
Let $\bar{x}$ be an $\varepsilon$-close solution of~\eqref{P1}.
\begin{enumerate}
\item There exists an optimal solution $x^\star$ of~\eqref{P0} satisfying 
\begin{equation*}
\norm{Ax^\star - A\bar{x}}_\infty \leq 2 m^{3/2} \norm{A}_{\infty} + \varepsilon.
\end{equation*}
\item The integral part $z^\star$ of $x^\star$ satisfies \label{pt:epsilon-2}
\begin{equation*}
\norm{A z^\star - A \bar{x}}_{\infty} \leq 3 m^{3/2} \norm{A}_{\infty} + \varepsilon.
\end{equation*}
\end{enumerate}
\end{corollary}
\begin{proof}
We start by estimating the distance from $A \bar{x}$ to $A \hat{x}$ for an optimal solution $\hat{x}$ of~\eqref{P1}.
We have
\begin{align*}
\| b - A\bar{x} \|_2^2 &= \| b - A \hat{x} + A\hat{x} - A \bar{x} \|_2^2 \\
&= \| b-A \hat{x} \|_2^2 + \| A \hat{x} - A \bar{x} \|_2^2 + 2 (b - A\hat{x})^\intercal (A\hat{x} - A\bar{x}),
\end{align*}
where the last term is non-negative by Lemma~\ref{lem:sep}.
Rearranging terms, we obtain
\begin{equation*}
\| A\hat{x} - A\bar{x} \|_2^2 = \| b - A \bar{x} \|^2_2 - \| b - A \hat{x} \|^2_2 - 2(b-A\hat{x})^\intercal (A\hat{x} - A\bar{x}) 
\leq \varepsilon^2.
\end{equation*}
Applying the triangle inequality and combining the above estimate with Theorem~\ref{thm:proximity}, we have
\begin{align*}
\|Ax^\star - A\bar{x}\|_{\infty} &\leq \|Ax^\star - A\bar{x}\|_{2} \leq \|Ax^\star - A\hat{x} \|_{2} + \|A\hat{x} - A\bar{x}\|_{2} \\
&\leq 2 m^{3/2} \norm{A}_{\infty} + \varepsilon.
\end{align*}
For Part~\ref{pt:epsilon-2},
recall that $x^\star - z^\star = f^\star$ with $\|f^\star\|_0 \leq m$, implying the inequality $\|A(x^\star -z^\star)\|_{\infty} \leq m \|A\|_{\infty}$.
We obtain
\begin{align*}
\|Az^\star - A\bar{x} \|_{\infty} &= \|Az^\star - Ax^\star + Ax^\star - A\bar{x}\|_\infty \\
&\leq \|A(z^\star - x^\star)\|_{\infty} + \|Ax^\star - A\bar{x}\|_\infty \\
&\leq m \norm{A}_{\infty} + 2 m^{3/2} \norm{A}_{\infty} + \varepsilon \leq 3 m^{3/2} \norm{A}_{\infty} + \varepsilon.
\end{align*}
\end{proof}

We next outline the decomposition $x^\star = z^\star + f^\star$.
If we chose a strict decomposition $f^\star = x^\star - \lfloor x^\star \rfloor$ we would have to guess all sets $\cF$ of cardinality at most $m$.
When reconstructing $f^\star$ later on however, we also allow entries with index in $\cF$ to be integral. 
This allows us to guess only sets $\cF \subseteq [n]$ with $|\cF| = m$.
As a first step, we guess the support of $f^\star$ in a straightforward way.
\begin{lemma}
\label{obs:guessing-fractional-entries}
There are $n^m$ potentially different index sets $\supp (f^\star)$.
\end{lemma}
A canonical approach would be to search for the vector $f^\star$.
Then we run again into the problem that our objective is a non-linear objective, and hence $f^\star$ depends on $z^\star$.
This requires us to first search for an optimal $z^\star$ and then use continuous optimization techniques to compute $f^\star$.

We denote by $A_{\setminus f^\star}$ the matrix $A$ without the columns with index in $\supp(f^\star)$.
The next theorem shows that we can compute a small set $Z^\star$ of possible vectors for $z^\star$. 
\begin{theorem}
\label{thm:finding-z-star}
Let $\bar{x}$ be an $\varepsilon$-close solution to~\eqref{P1}.
If $\supp(f^\star)$ is fixed, we can compute a set $Z^\star \subseteq \{0,1\}^n$ of candidate vectors such that \ $x^\star = z^\star + f^\star$ with $z^\star \in Z^\star$.
This requires us to solve at most $(6 m^{3/2} \norm{A}_{\infty} + 2\varepsilon +1)^m$ linear integer programming problems.
\end{theorem}
\begin{proof}
We have  $Az^\star \in A\bar{x} + [-D_{\varepsilon},D_{\varepsilon}]^m \cap \Z^m$,
where $D_{\varepsilon} = 3 m^{3/2} \norm{A}_{\infty} + \varepsilon$  by Corollary~\ref{cor:proximity}.
For every $b^\star \in A \bar{x} + [-D_{\varepsilon},D_{\varepsilon}]^m \cap \Z^m$ we solve the integer feasibility problem
\[
A_{\setminus f^\star} y = b^\star, \quad
\sum_{i=1}^{n-m} y_i \leq \sigma-m, \quad 
y \in \{ 0, 1\}^{n-m}. 
\]
If it has a feasible solution $y$, we can insert zero entries according to $\supp(f^\star)$ and obtain a vector $z \in \{0,1\}^n$ that qualifies as the vector $z^\star$.
The set $Z^\star$ is the set of all extended vectors $z$.
\end{proof}
It remains to compose each $z^\star \in Z^\star$ with a vector $f^\star$.
This is accomplished by solving a series of least-square problems.
The reason why we proceed in this way is that it allows us to compute the exact vector $f^\star$ as opposed to an $\varepsilon$-close solution.
\begin{lemma}[Extension lemma]
For each $z \in Z^\star$ an optimal solution $f$ to $\min \{\|b - Az - Af \|_2 : \ \supp(f) \subseteq \supp(f^\star), 0 \leq f \leq \mathbf{1} \}$ can be computed in $\mathcal{O} (3^m m^3)$ arithmetic operations.
\end{lemma}
\begin{proof}
As $f_i = 0$ for $i \notin \supp (f^\star)$, we can restrict to the matrix $A_{f^\star} \in \Z^{m \times m}$ and solve the equivalent problem $\min \{\|b^\prime - A_{f^\star} g\|_2 : \ g \in [0,1]^m\}$ for $b^\prime \df b - Az$.
Without the variable bounds this is a least-square problem that can be solved in $\mathcal{O}(m^3)$ arithmetic operations.
Let $g^\star$ be an optimal solution.
We guess the sets $S_0 \df \{i: \ g^\star_i = 0\}$ and $S_1 \df \{i: \ g^\star_i = 1\}$, and afterwards solve the 
modified least-square problem $\min \{\|b - A_{f^\star} g\|_2 : \ g_i = 0 \ \forall i \in S_0, g_i = 1 \ \forall i \in S_1\}$.
If the solution $g$ is in $[0,1]^n$, its extension $f \in [0,1]^n$ qualifies as $f^\star$.
In the end, we pick the best among all feasible extensions.
As there are $3^m$ guesses, this finishes the proof.
\end{proof}
This completes the presentation of the main steps to prove Theorem~\ref{thm:main-1}.
In fact, in order to obtain an optimal solution to~\eqref{P0} one proceeds as follows.
We first guess the set $\supp(f^\star)$, determine the set $Z^\star$ and compute for every $z^\star \in Z^\star$ an optimal vector $f^\star$.
The best of all those solutions solves~\eqref{P0}.
As a last technicality, we have to show how to find an $\varepsilon$-close solution $\bar{x}$ for which we fall back on~\cite[Chap.\ 8]{Nesterov1994InteriorpointPA}.
\begin{lemma}[{\cite[Chap.\ 8]{Nesterov1994InteriorpointPA}}]
We can find a $\sqrt{m}\|A\|_{\infty}$-close solution for~\eqref{P1} in $\mathcal{O} \left( n^{7/2} \ln \left( n^2 \sigma \|b\|_1  \right) \right)$ arithmetic operations.
\end{lemma}
\begin{proof}
We apply the results presented in~\cite[Chap.\ 8]{Nesterov1994InteriorpointPA} that depend on several parameters.
Let $P \df \{x \in [0,1]^n: \ \|x\|_1 \leq \sigma \}$ denote the feasible region of~\eqref{P1} and $\hat{x}$ an optimal solution.
We first need to estimate
\begin{equation*}
\mathcal{D} \df \max \{ \|b-A y \|_2^2 - \|b-A\hat{x}\|_2^2: \ y  \in P \}.
\end{equation*}
For any $y \in P$ we can estimate
\begin{align*}
\| b -  A y \|_2^2 - \|b - A\hat{x} \|^2_2 &= \|A y \|_2^2 - \|A \hat{x} \|_2^2 + 2b^\intercal A (\hat{x} - y) \\
&\leq \sigma^2 m \| A \|_{\infty}^2 + 4 \| b \|_1 \sigma \|A\|_{\infty} \\
&\leq 4 \sigma^2 m \| A \|_{\infty}^2 (\|b\|_1 + 1),
\end{align*}
resulting in $\mathcal{D} \leq 4 \sigma^2 m \| A \|_{\infty}^2 (\|b\|_1 + 1)$.
As the initial point in the interior of $P$ that is required in~\cite[Chap.\ 8]{Nesterov1994InteriorpointPA} we choose $w \df \tfrac{\sigma}{n + \sigma} \cdot \mathbf{1}$ where $\mathbf{1}$ denotes the all-ones vector.
Next we estimate the \emph{asymmetry coefficient} 
\[
\alpha (P:w) \df \max \{ t: \ w + t(w-P) \subseteq P \}.
\]
Since $[0,\tfrac{\sigma}{n}]^n \subseteq P \subseteq [0,1]^n$, for $t = \tfrac{\sigma}{n}$ we obtain
\[
w + t ( w-P ) \subseteq  w + t ( w - [0,1]^n ) 
= \left[ 0 , \tfrac{\sigma}{n} \right]^n \subseteq P,
\]
thus $\alpha (P:w) \geq \tfrac{\sigma}{n}$.
By~\cite[Chap.\ 8, Eq. 8.1.5]{Nesterov1994InteriorpointPA} we can compute a feasible solution $\bar{x}$ of \eqref{P1} satisfying 
$\|b-A\bar{x}\|_2^2-\|b-A\hat{x}\|_2^2 \leq \delta \mathcal{D}$ in $O(1)(2n+1)^{1.5}n^2\ln \left(\frac{2n+1}{\alpha(P:w) \delta}\right)$ arithmetic operations. 
Finally, by choosing $\delta = \frac{1}{4\sigma^2(\norm{b}_1+1)}$ finding a $\sqrt{m} \|A\|_{\infty}$-close solution
takes $\mathcal{O} \left( n^{7/2} \ln \left( n^2 \sigma \|b\|_1  \right) \right)$ arithmetic operations.
\end{proof}


\section{Extension}
\label{sec:conclusion}
A natural generalization of our problem is to consider arbitrary upper bounds $u_i>0$, i.e. 
\begin{equation}
\label{l0-problem arbitrary bounds}
\tag{$P_0^\prime$}
\min_x \| Ax-b \|_2
\text{ subject to } \norm{x}_0 \leq \sigma \text{ and } 0 \leq x_i \leq u_i \text{ for all } i \in \lbrack n \rbrack.
\end{equation}
The natural convex relaxation of~\eqref{l0-problem arbitrary bounds} is given by:
\begin{equation}
\label{l1-relaxation arbitrary bounds}
\tag{$P_1^\prime$}
\min_x \| Ax-b \|_2
\text{ subject to }  \sum_{i=1}^n \frac{x_i}{u_i} \leq \sigma \text{ and } 0 \leq x_i \leq u_i \text{ for all } i \in \lbrack n \rbrack.
\end{equation}
The results of Sections~\ref{sec:proximity} and~\ref{sec:algorithm} extend to this generalization in a straight-forward manner.
For the algorithm it implies that the number of arithmetic operations increases by an additional factor of $\norm{u}_{\infty}^m$.
The reason is the core of our approach:
The proximity bound between optimal solutions for~\eqref{l0-problem arbitrary bounds} and~\eqref{l1-relaxation arbitrary bounds} respectively, increases by this factor.
The proximity bound must however depend on $\| u \|_{\infty}$ as the following example shows:

Let $n$ and $u$ be even, non-negative integers.
Set $A  \df \mathbb{1}, \sigma \df \frac{n}{2}$ and $b=\frac{u}{2}\mathbf{1}$ where $\mathbf{1}$ denotes the all-ones vector.
It can easily be checked that $\hat{x}=\frac{u}{2}\mathbf{1}$ is optimal for~\eqref{l1-relaxation arbitrary bounds} while
\begin{equation*}
x^{\star}_i=\begin{cases} \frac{u}{2} , i \in \lbrack \sigma \rbrack \\
0, i \in \lbrack n \rbrack \setminus \lbrack \sigma \rbrack \end{cases}
\end{equation*}
is optimal for~\eqref{l0-problem arbitrary bounds}.
This shows that any approach aiming for a logarithmic dependency on $\norm{u}_{\infty}$ requires techniques that are different from the ideas presented in this paper.

\subsubsection{Acknowledgements} 
The second and third author acknowledge support by the Einstein Foundation Berlin.

%
%
%
 \bibliographystyle{splncs04}
%

\bibliography{bibliography}
\end{document}